\documentclass{amsart}

\usepackage{amsmath}
\usepackage{amssymb}
\usepackage{amscd}
\usepackage{hyperref}

\newtheorem{theorem}{Theorem}[section]
\newtheorem{thm}[theorem]{Theorem}
\newtheorem{cor}[theorem]{Corollary}
\newtheorem{lem}[theorem]{Lemma}
\newtheorem{prop}[theorem]{Proposition}

\theoremstyle{definition}
\newtheorem{defn}[theorem]{Definition}

\newtheorem{rem}[theorem]{Remark}

\newtheorem{conj}[theorem]{Conjecture}

\theoremstyle{remark}

\newcommand{\mbb}{\mathbb}
\newcommand{\QQ}{\mbb{Q}}

\newcommand{\ZZ}{\mbb{Z}}

\newcommand{\RR}{\mbb{R}}

\newcommand{\PP}{\mbb{P}}

\newcommand{\FF}{\mbb{F}}

\newcommand{\mc}{\mathcal}

\newcommand{\mcC}{\mc{C}}

\newcommand{\mcE}{\mc{E}}

\newcommand{\mcL}{\mc{L}}

\newcommand{\mcN}{\mc{N}}

\newcommand{\mcT}{\mc{T}}

\newcommand{\mcX}{\mc{X}}

\newcommand{\OO}{\mc{O}}

\newcommand{\p}{\mathbb{P}^{1}}

\newsavebox{\sembox}
\newlength{\semwidth}
\newlength{\boxwidth}

\newsavebox{\semrbox}
\newlength{\semrwidth}
\newlength{\boxrwidth}

\newcommand{\GWpp}{\langle [pt], [pt], \ldots\rangle_{0, \beta}^{X}}
\newcommand{\GWp}{\langle [pt], \ldots\rangle_{0, \beta}^{X}}
\newcommand{\tx}{\tilde{X}}

\title
{Symplectic geometry of rationally connected threefolds}

\author[Tian]{Zhiyu Tian}
\address{
Department of Mathematics \\
Stony Brook University \\ 
Stony Brook, NY, 11794}
\email{ztian@math.sunysb.edu}

\date{\today}

\begin{document}


\begin{abstract}
 We study symplectic geometry of rationally connected $3$-folds. The first result shows that rationally connectedness is a symplectic deformation invariant in dimension $3$. If a rationally connected $3$-fold $X$ is Fano or $b_2(X)=2$, we prove that it is symplectic rationally connected, i.e. there is a non-zero Gromov-Witten invariant with two insertions being the class of a point. Finally we prove that many rationally connected $3$-folds are birational to a symplectic rationally connected variety.
\end{abstract}


\maketitle



\section{Introduction}
In this paper we study the symplectic geometry of smooth projective rationally connected $3$-folds over the complex numbers. Let's first recall the relevant definitions.
\begin{defn}
A variety $X$ is called \emph{rationally connected} if two general points in $X$ can be connected by a rational curve.
\end{defn}
 
A related notion is uniruledness.

\begin{defn}
A variety $X$ is called \emph{uniruled} if there exists a rational curve through a general point.
\end{defn}

Therefore these varieties contain lots of rational curves. Of the most importance are the following two types of curves:
\begin{defn}
Let $X$ be a smooth variety. A curve $f: \p \rightarrow X$ is called \emph{free} (resp. \emph{very free}) if $f^*T_X\cong \oplus \OO_{\p}(a_i)$ with $a_i \geq 0$ (resp. $a_i \geq 1$). 
\end{defn}

A smooth projective variety is uniruled (resp. rationally connected) if and only if there is a free (resp. very free) curve.

One motivation of this paper is the following theorem, proved independently by Koll{\'a}r and Ruan.
\begin{thm}[\cite{KollarUni}, ~\cite{RuanUni}]\label{thm:KoRu}
Let $X$ be a smooth projective uniruled variety. Then there is a non-zero Gromov-Witten invariant of the form $\GWp$.
\end{thm}

If $X$ and $X'$ are two smooth projective varieties, then they can also be considered as symplectic manifolds with symplectic form $\omega$ and $\omega'$ given by the polarizations. We say that $X$ and $X'$ are \emph{symplectic deformation equivalent} if there is a family of symplectic manifolds $(X_t, \omega_t)$ diffeomorphic to each other such that $(X_0, \omega_0) \cong (X, \omega)$ and $(X_1, \omega_1) \cong (X', \omega')$. Since Gromov-Witten invariants are symplectic deformation invariants, Koll{\'a}r and Ruan's result implies that uniruledness is a symplectic deformation invariant. Then it is natural to ask if rationally connectedness is also a symplectic deformation invariant. In fact, Koll{\'a}r conjectured the following:
\begin{conj}[Koll{\'a}r]\label{conj:Kollar}
Let $X$ and $X'$ be two smooth projective varieties which are symplectic deformation equivalent. Then $X$ is rationally connected if and only if $X'$ is.
\end{conj}

The first evidence is the following theorem of Voisin~\cite{VoisinRC}.
\begin{thm}[~\cite{VoisinRC}]
Let $X$ and $X'$ be two smooth projective $3$-folds which are symplectic deformation equivalent . If $X$ is Fano or rationally connected with $b_2(X)=2$, then $X'$ is also rationally connected. 
\end{thm}

Our first result is the proof of the above conjecture in dimension $3$.

\begin{thm}\label{thm:Invariance}
The conjecture \ref{conj:Kollar} is true in dimension $3$.
\end{thm}

The idea of the proof ( which is motivated by the proof in~\cite{VoisinRC}) is the following. It suffices to show that the \emph{maximal rationally connected quotient (MRC-quotient)} of $X'$ is a point. By the result of Koll{\'a}r and Ruan
~\ref{thm:KoRu}, the MRC quotient is either a surface, a curve or a point. For topological reasons, it cannot be a curve. If it is a surface, then $X$ is birational to a smooth $3$-fold $Y$ which is a conic bundle over a smooth rational surface. On this new threefold, we can find a non-zero genus zero Gromov-Witten invariant of the form $\langle [C], \ldots \rangle_{0, \beta}^{X'}$, where $C$ is the Poincare dual of the curve class of a general fiber. By weak factorization, we can factorize the birational map from $Y$ to $X$ by a sequence of blow-ups and blow-downs. Then using the ideas developed in ~\cite{MP} and ~\cite{HLR}, we show that there is a similar non-zero descendant Gromov-Witten invariant on $X$, hence on $X'$. So the MRC quotient of $X'$ cannot be a surface.

This is clearly different from Theorem ~\ref{thm:KoRu}. One would like to know if there is a non-zero Gromov-Witten invariant of the form $\GWpp$ on a smooth projective rationally connected variety. One dimensional rationally connected variety is just $\PP^1$. So this is true. It is also easy to prove this in dimension $2$ (c.f. Proposition~\ref{prop:surface}). In general, this is very difficult since the moduli space might be reducible and there might be components whose dimension are higher than the expected dimension. Then one has to do the computation on the \emph{virtual fundamental class} in order to get the Gromov-Witten invariants. The components of higher dimensions can contribute negatively to the invariant, thus making it $0$. 

Our second theorem addresses this question in some special cases.

\begin{thm}\label{thm:SymRC}
Let $X$ be a smooth projective rationally connected $3$-fold. If $X$ is Fano or $b_2(X)=2$, then there is a non-zero Gromov-Witten invariant of the form $\GWpp$. Here $b_2(X)$ is the second Betti number of $X$.
\end{thm}

Next we would like to mention the so called "symplectic birational geometry program". We start with some definitions.

\begin{defn}
A symplectic manifold is \emph{symplectic uniruled} (resp. \emph{symplectic rationally connected (SRC)}) if there is a non-zero Gromov-Witten invariant of the form $\GWp$ (resp. $\GWpp$).
\end{defn}

There are two basic questions about these definitions: symplectic birational invariance and whether or not a smooth projective uniruled (rationally connected) variety is symplectic uniruled (SRC). For symplectic uniruledness, the answer is positive (~\cite{KollarUni}, ~\cite{RuanUni}, ~\cite{HLR}). It is not known if symplectic rationally connectedness is a (symplectic) birational invariant, although we do expect this to be true. And as noted above, it is also not known if rationally connected projective manifolds are SRC.

One could try to prove that a rationally connected variety is SRC by showing the birational invariance and try to find in each birational class a "good" representative which is SRC. In this paper, we try to carry out the second part in dimension $3$.

By the minimal model program in dimension $3$, every rationally connected variety is birational to one of the following
\begin{enumerate}
\item a conic bundle over a rational surface,
\item a fiberation over $\PP^1$ with general fiber a Del Pezzo surface, or
\item a $Q$-Fano threefold.
\end{enumerate}

Here is our third theorem.

\begin{thm}\label{thm:SRC}
In the first two cases listed above, there is a resolution which is SRC. In the last case, if the smooth locus is rationally connected, then there is a resolution which is SRC.
\end{thm}

Thus the first two cases are completely settled. Our proof of the second and third case also works in a slightly more general context. The reader is referred to Section ~\ref{sec:delpezzo} and ~\ref{sec:QFano} for more precise statements. In general, the condition of the last case is very difficult to verify. However in this paper we note that Gorenstein $Q$-Fano $3$-folds satisfy this condition.

Finally, we mention a related question. By the theorem of Graber-Harris-Starr~\cite{GHS03}, a rationally connected fiberation over a curve always has a section. As a corollary, the total space of a rationally connected fiberation over a rationally connected variety is itself rationally connected. It would be interesting to know if a similar result holds in the symplectic category. A first step might be to analyze the case of a rationally connected fiberation over a rational curve. If a general fiber is $\PP^1$, then it is true (c.f. Proposition~\ref{prop:surface}). This is already quite difficult to analyze when the fiber dimension is $2$. The second part of Theorem~\ref{thm:SRC} is also an attempt to analyzed a special case: a fiberation in Del Pezzo surfaces coming from a contraction. We prove that there is a section on the resolution which gives the non-zero Gromov-Witten invariant with two insertions being the class of a point. Understanding this symplectic version of Graber-Harris-Starr theorem may help solving the question of symplectic rationally connectedness since in many cases, the minimal model program produces birational models as a fiberation. One may try to compare the Gromov-Witten invariants on the total space and those of a fiberation over a curve in the base(c.f. the proof of Theorem~\ref{thm:conic}).

\textbf{Acknowledgments:} The author would like to thank Mingmin Shen for showing him his Ph.D. thesis, Jason Starr for inspirations and encouragement, Aleksey Zinger for helping him understand Gromov-Witten invariants and the degeneration formula .

\section{Symplectic invariance for rationally connected $3$-folds}

\subsection{Descendant GW-invariants, Relative GW-invariants, and Degeneration formula}
In this section we recall some variants of Gromov-Witten invariants.

\begin{defn}
Let $\mcL_i$ be the line bundle on the moduli stack $\overline{\mathcal{M}}_{0, n}^{X, \beta}$ whose fiber over each point $(C, p_1, \ldots, p_n)$ is the restriction of the cotangent line bundle of $C$ to $p_i$. Let $\psi_i$ be the first Chern class of $\mcL_i$. Then the descendant Gromov-Witten invariant is defined as 
\[
\langle \tau_{k_1}A_1, \ldots, \tau_{k_m}A_m \rangle^{X}_{0, \beta}=\int_{[\overline{\mathcal{M}}_{0, n}^{X, \beta}]^{\text{virt}}} \Pi_i \psi_i^{k_i} ev_i^*A_i,
\]
where $ev_i$ is the evaluation map given by the $i$-th marked point, $A_i \in H^*(X, \QQ)$.
\end{defn}

The relative Gromov-Witten invariants are first introduced in the symplectic category by Li-Ruan~\cite{LiRuan} and then in the algebraic category by Jun Li~\cite{LiJun1},~\cite{LiJun2}. We will not recall the precise definition here since it is not needed. The reader should refer to the above-mentioned papers for more details.

Intuitively, the relative Gromov-Witten invariants count the number of stable maps meeting certain constraints and having prescribed tangency condition with a given divisor. Let $X$ be a smooth projective variety and $D \subset X$ be a smooth divisor. Fix a curve class $\beta$ such that $D \cdot \beta =n \geq 0$ (relative Gromov-Witten invariants are not defined if $D \cdot \beta < 0$). Also choose a partition $\{m_i\}$ of $n$. Then the relative Gromov-Witten invariants count the number of stable maps $f: (C, p_1, \ldots, p_r, q_1, \ldots, q_r) \rightarrow X$ with $r+s$ marked points such that the first $r$ points (absolute marked points) are mapped to cycles in $X$ and the last $s$ points (relative marked points) are mapped to some cycles in $D$ with $f^*D=\sum m_i q_i$. We can also define descendant relative Gromov-Witten invariants. We write such invariants as 
\[
\langle \tau_{k_1}A_1, \ldots, \tau_{k_r}A_r \vert (m_1, B_1), \ldots, (m_s, B_s)\rangle^{(X, D)}_{0, \beta}
\] 
where $A_i \in H^*(X, \QQ), B_j \in H^*(D, \QQ)$. We also use the abbreviation 
\[
\langle \Gamma \{(d_i, A_i)\} \vert {\mcT}_k \rangle^{X, D}_{\beta}
\]
 following ~\cite{HLR}. In the degeneration formula, we have to consider stable maps from disconnected domains. The corresponding relative invariants are defined to be the product of those of stable maps from connected domains. Such invariants are denoted by 
\[
\langle \Gamma^\bullet \{(d_i, A_i)\} \vert \mathcal{T}_k \rangle^{X, D}_{\beta}.
\]
Finally we note that we can represent these invariants by decorated dual graphs (c.f. Sec. 3.2 in~\cite{HLR}).

Now we describe the degeneration formula. Let $W \rightarrow S$ be a projective morphism from a smooth variety to a pointed curve $(S, 0)$ such that a general fiber is smooth and connected and the fiber over $0$ is the union of two smooth irreducible varieties ($W^+, W^-$) intersecting transversely at a smooth subvariety $Z$. Let $A_i$ be cohomology classes in a general fiber. Assume that the specialization of $A_i$ in $W_0$ can be written as $A_i(0)=A_i^++A_i^-$, where $A_i^+ \in H^*(W^+, \QQ)$ and $A_i^- \in H^*(W^-, \QQ)$. Let $\{\beta_i\}$ be a self-dual basis of $Z$. Also let $\mcT_k=\{(t_j, \beta_{a_j})\}$ be a weighted partition and $\check{\mcT}_k=\{(t_j, \check{\beta}_{a_j})$ be the dual partition, i.e. $\check{\beta}_{a_j}$ is the dual of $\beta_{a_j}$. Then the degeneration formula expresses the Gromov-Witten invariants of a general fiber in terms of the relative Gromov-Witten invariants of the degeneration in the following way:
\[
\langle \Pi_i \tau_{d_i} A_i \rangle^{W_t}=\sum \Delta(\mcT_k) \langle \Gamma^\bullet \{(d_i, A_i^+)\} \vert \mathcal{T}_k \rangle^{W^+, Z} \langle \Gamma^\bullet \{(d_i, A_i^-)\} \vert \check{\mathcal{T}}_k \rangle^{W^-, Z},
\]
where the summation is taken over all possible degenerations and 
\[
\Delta(\mcT_k)=\Pi_j t_j Aut(T_k).
\]

In this paper we are mainly interested in the following special case of such degenerations: the deformation to the normal cone. Namely, let $X$ be a smooth projective variety and $S \subset X$ be a smooth subvariety. Then we take $W$ to be the blow up of $X \times A^1$ with blow up center $S \times 0$. In this case, $W^- \cong \tx$, the blow up of $X$ along $S$, and $W^+\cong \PP_S(\OO \oplus N_{S/X})$.

\subsection{A partial ordering}
Let $X$ be a smooth projective $3$-fold and $S\subset X$ be a smooth subvariety of codimension $k$. Denote by $\tilde{X}$ the blow up of $X$ along $S$ and by $E$ the exceptional divisor. Here we allow $S$ to be a codimension $1$ subvariety, i.e. a divisor, and in this case $\tilde{X} \cong X, E \cong S$.

Let $\theta_1=1, \theta_2, \ldots, \theta_{m_S}=\omega \in H^*(S, \QQ)$ a self dual basis of $S$, where $1$ (resp. $\omega$) is the generator of the degree $0$ (resp. $2(3-k)$) cohomology. We now describe a self dual basis of $E$, where $E=\PP_S(N_{S/X})$ is a $\PP^{k-1}$ bundle over $S$. Let $[E]$ be the first Chern class of the relative $\OO(1)$ bundle over $\PP_S(N_{S/X})$. If $k=2$, i.e. $S$ is a smooth curve in $X$, then $\pi_S: E \rightarrow S$ is a ruled surface over $S$. We have $[E] \cdot [E]=d$ on $E$. In this case, take $\lambda=[E]-\frac{d}{2}\pi_S^*\omega$. Otherwise just take $\lambda$ to be $[E]$. Then the cohomology classes
\[
\pi_S^* \theta_i \cup \lambda^j, ~1 \leq i \leq m_S, ~0 \leq j \leq k-1
\]
form a self dual basis of $E$. Denote them by $\Theta=\{\delta_i \}$.

\begin{rem}
In the paper~\cite{HLR}, the authors claim $\pi_S^* \theta_i \cup [E]^j$ to be self dual, which is not true if $N_{S/X}$ is not a trivial bundle over $S$. However, this has been fixed and the proof is essentially the same since only the degree of the $[E]$ part is important in the proof. The authors take a different self dual basis.\footnote{Private communication.}
\end{rem}

\begin{defn}
A standard (relative) weighted partition $\mu$ is a partition
\[
\mu=\{(\mu_1, \delta_{d_1}), \ldots, (\mu_{l(\mu)}, \delta_{d_{l(\mu)}}\},
\]
where $\mu_i$ and $d_i$ are positive integers with $d_i \leq k m_S$. $l(\mu)$ is called the length of the partition.
\end{defn}

\begin{defn}
For $\delta=\pi_S^* \theta \cup \lambda^j \in H^*(E, \QQ)$, define
\[
\deg_S(\delta)=\deg \theta, deg_f(\delta)=2j.
\]

For a standard weighted partition $\mu$, define
\[
\deg_S(\mu)=\sum_{i=1}^{l(\mu)} \deg_S \delta_{d_i}, 
\]
\[
\deg_f(\mu)=\sum_{i=1}^{l(\mu)}\deg_f(\delta_{d_i}).
\] 
\end{defn}

\begin{defn}
We define a partial ordering on the set of pairs $(m, \delta)$ where $m \in  \ZZ_{>0}$ and $\delta \in H^*(E, \QQ)$ as follows:
\[
(m, \delta) > (m', \delta')
\]
if
\begin{enumerate}
\item $m>m'$, or
\item $m=m'$ and $\deg_S(\delta) > \deg_S(\delta')$, or
\item equality in the above and $\deg_f(\delta) > \deg_f(\delta')$.
\end{enumerate}

\end{defn}

\begin{defn}
A lexicographic ordering on weighted partitions is defined as following:
\[
\mu > \mu'
\]
if after we place the pairs of $\mu$ and $\mu'$ in decreasing order, the first pair for which $\mu$ and $\mu'$ are not equal is larger for $\mu$.
\end{defn}

Let $\sigma_1, \ldots, \sigma_{m_X}$ be a set of basis of $H^*(X, \QQ)$. Then the set of cohomology classes
\[
\gamma_j=\pi^* \sigma_j, 1 \leq j \leq m_X,
\]
\[
\gamma_{j+m_X}=\iota_*(\delta_j), 1 \leq j \leq k m_S
\]
is a set of $\QQ$-basis of $\tilde{X}$, where $\iota : E \rightarrow \tilde{X}$ is the inclusion and $\iota_*$ is the induced Gysin map.

\begin{defn}
A connected standard relative Gromov-Witten invariant of $(\tilde{X}, E)$ is of the form
\[
\langle \omega \vert \mu \rangle^{\tilde{X}, E}_{0, A}=\langle \tau_{k_i}\gamma_{L_1}, \ldots, \tau_{k_n}\gamma_{L_n} \vert \mu \rangle^{\tilde{X}, E}_{0, A},
\]
where $A$ is an effective curve class, $\mu$ is a standard weighted partition with $\sum \mu_j = E \cdot A$, $\gamma_{L_i}=\pi^* \sigma_{L_i}$.
\end{defn}

We write $\Gamma(\omega) \vert \mu$ for the graph of such invariants. Here is the partial ordering on such graphs.

\begin{defn}
\[
\Gamma(\omega) \vert \mu <\Gamma(\omega') \vert \mu'
\]
if
\begin{enumerate}
\item $\pi_*(A) < \pi_*(A')$, i.e. the difference $\pi_*(A')-\pi_*(A)$ is an effective curve class in $X$.
\item equality in (1) and the arithmetic genus satisfies $g' < g$,
\item equality in (1) and (2) and $\Vert \omega' \Vert < \Vert \omega\Vert$,
\item equality in (1)-(3) and $\deg_S(\mu') >\deg_S( \mu)$,
\item equality in (1)-(4) and $\mu' > \mu$.
\end{enumerate}
where $\Vert \omega \Vert$ is the number of insertions of $\omega$.
\end{defn}

We have the following observation:
\begin{lem}
Given a standard relative invariants, there are only finitely many standard relative invariants smaller than it in the partial ordering defined above.
\end{lem}
\begin{proof}
For a curve class $[A]$ of $\tilde{X}$, there are only finitely many curve classes lower than it. Thus only finitely many relative invariants lower than the given invariants. In our definition, we are comparing the curve classes by their image in $X$. It is easy to see that two different curve classes in $\tilde{X}$ defines the same curve class in $X$ if and only if the difference is a multiple of $L$, where $L$ is a line or a ruling in the exceptional divisor. So we have to consider all the possible standard relative invariants associated to $A+kL$. But $k$ is bounded below since $A+kL$ has to be an effective curve class. It is also bounded above since $E \cdot (A+kL) \geq 0$ and $E \cdot L=-1$.

\end{proof}

\subsection{From relative to absolute}

In this subsection, we discuss how to associate an absolute invariant of $X$ to a relative invariant of $\tilde{X}$.
\begin{defn}
For a relative insertion $(m, \delta)$ with $\delta=\pi_S^*\theta_i \cup \lambda^j$, we associate to it the absolute insertion $\tau_{d(m, \delta)}(\tilde{\delta})$, where
\[
\tilde{\delta}=\iota_*(\theta_i),
\]
\[
d(m, \delta)=km-k+j.
\]
Given a weighted partition $\mu=\{(\mu_i, \delta_{k_i})\}$, we define 
\[
d_i(\mu)=d(\mu_i, \delta_{k_i})=k \mu_i-k+\frac{1}{2}\deg_f(\delta_{k_i}),
\]
\[
\tilde{\mu}=\{\tau_{d_1(\mu)}(\tilde{\delta}_{k_1}), \ldots, \tau_{d_{l(\mu)}(\mu)}(\tilde{\delta}_{k_{l(\mu)}})\}.
\]
Given a standard relative invariant $\langle \Gamma^\bullet(\omega) \vert \mu\rangle^{\tilde{X}, E}$, we define the absolute descendant invariant associated to the relative invariant to be 
\[
\langle \tilde{\Gamma}^\bullet(\omega, \tilde{\mu})\rangle^X
\]
Here all the insertions $\omega$ in the relative invariants are of the form $\pi^* \sigma_i$. And the corresponding insertions in the absolute invariants are just $\sigma_i$.
\end{defn}

\begin{defn}
An absolute descendant invariant of $X$ is called a \emph{colored absolute descendant invariant relative to $S$} if it can be written in the form $\langle \Gamma(\omega, \tilde{\mu})\rangle$ such that each insertion in $\omega$ is of the form $\tau_{d_i}\sigma_i$ and each insertion in $\tilde{\mu}$ is of the form $\tau_d \tilde{\delta}_k$.
\end{defn}

\begin{defn}
If $k=1$, then a colored absolute descendant invariant of $X$ relative to $S$ is called \emph{admissible} if $\sum \mu_j = E \cdot A$.
\end{defn}

The following lemma is essentially Lemma 5.14 in ~\cite{HLR}. Note that in their paper they only consider the case of primary Gromov-Witten invariants. But the proof is actually the same.
\begin{lem}\label{lem:bijection}
If $\mu \neq \mu'$, then $\tilde{\mu} \neq \tilde{\mu'}$. Therefore there is a natural bijection between the set of colored weighted absolute graphs relative to $S$ and the set of weighted relative graphs in $\tilde{X}$ relative to $E$ if $k>1$. The same is true if we restrict to the admissible ones if $k=1$.
\end{lem}

\begin{rem}
Notice that different relative invariants may give the same absolute invariants. But these absolute invariants are different as colored absolute invariants.
\end{rem}

Finally, let $C$ be a curve in $\tx$ such that $C$ does not intersect $E$. Then $C$ gives a curve in $X$, also denoted by $C$. Let $I$ be the partially ordered set of standard weighted relative graph $\Gamma^\bullet ([C], \omega )\vert \mu$. We can form an infinite dimensional vector space $\RR^I_{\tilde{X}, E}$ whose coordinates are ordered in the same way as the partial ordering in $I$. A standard weighted relative invariant $\langle \Gamma^\bullet([C], \omega) \vert \mu \rangle^{\tilde{X}, E}$ gives rise to a vector in $\RR^I_{\tilde{X}, E}$. By Lemma~\ref{lem:bijection}, $I$ is also the set of colored standard weighted absolute graphs relative to $S$. Thus we also have an infinite dimensional vector space $\RR^I_{X, S}$. Similarly, an absolute invariant $\langle \Gamma^\bullet([C], \omega, \tilde{\mu})\rangle^X$ gives a vector in this vector space.

\subsection{The Correspondence}
In this section we sketch the proof of the following theorem.
\begin{thm}\label{thm:Correspondence}
Let $\pi: \tilde{X} \rightarrow X$ be the blow up of a $3$-fold along a smooth center $S$. Then the map
\[
A: \RR^I_{\tilde{X}, E} \rightarrow \RR^I_{X, S}
\]
which associate to a standard relative invariant a colored absolute invariant is an invertible lower triangular linear map.
\end{thm}

\begin{proof}
We take $W$ to be the blow up of $X \times A^1$ along the smooth subvariety $S\times 0$. Then $W^-\cong \tilde{X}, W^+ \cong \PP_S(\OO\oplus N_{S/X})$ and they intersect transversely at $E$. We will apply the degeneration formula to it. 

We start with a connected standard weighted relative invariant 
\[
\langle \Gamma([C], \omega) \vert \mu\rangle^{(\tilde{X}, E)}_{0, \beta}
\]
 with vertex decorated by $(0, \beta)$. Then the associated absolute invariant is 
\[
\langle \Gamma([C], \omega, \tilde{\mu}\rangle^{X}_{0, \pi_*(\beta)}.
\]
 In order to apply the degeneration formula, we have to specify the specialization of the cohomology classes. Since $C$ does not intersect $E$, we may specialize $C$ to lie entirely in $W^-$, i.e. we set $[C]^-=[C]$ and $[C]^+=0$. The Poincar{\'e} dual of the cohomology classes in $\tilde{\mu}$ are supported in $S$. So we specialize these cohomology classes to the $W^+$ side and set the cohomolgy in $W^-$ side to be zero. Finally, classes in $\omega$ are of the form $\sigma_i$. Therefore we set $\sigma_i^-=\gamma_i=\pi^* \sigma_i$ with appropriate classes $\sigma^+$ in $W^+$ side.

Then the degeneration formula in this case gives:
\begin{align*}
&\langle [C], \omega, \tilde{\mu} \rangle^{X}_{\pi_*(\beta)}\\
=&\sum \langle \Gamma^\bullet_-([C], \omega_1) \vert \eta \rangle^{(\tilde{X}, E)} \Delta(\eta) \langle \Gamma^\bullet_+(\omega_2, \tilde{\mu})\vert \check{\eta}\rangle^{\PP_S(\OO \oplus N_{S/X}), E}.
\end{align*}

We need to show that $\langle \Gamma([C], \omega) \vert \mu \rangle$ is the largest (non-zero) term in the right hand side. Note that we have
\[
\langle \Gamma^\bullet_-([C], \omega) \vert \mu \rangle^{(\tilde{X}, E)} \Delta(\eta) \langle  \Gamma_+^\bullet(\tilde{\mu})\vert \check{\mu}\rangle^{\PP_S(\OO \oplus N_{S/X}), E}
\]
on the right hand side. First we will show the coefficient $\langle \Gamma^\bullet_+(\tilde{\mu})\vert \check{\mu}\rangle^{\PP_S(\OO \oplus N_{S/X}), E}$ is non-zero. This is basically step II in the proof of Theorem 5.15 in ~\cite{HLR}. One can check that under our choice of the self dual basis, the coefficient is the product of relative invariants
\[
\langle \tau_{nd-1-j}[pt] \vert H^j \rangle^{\PP^k, \PP^{k-1}}_{0, dL}
\]
with $j=\deg_f(\delta_{k_i})$ and $H$ is the hyperplane class. These type of invariants are computed in ~\cite{HLR} via virtual localization and are shown to be non-zero. So the diagonal of the linear map $A_S$ is non-zero.

We remark that this is the only step in ~\cite{HLR} where the form of self dual basis is important since we need to know the diagonal is non-zero. For the rest part, the proof proceeds exactly as the proof of Theorem 5.15 in ~\cite{HLR} since only the property of being self dual is needed.

\end{proof}

\subsection{Birational invariance}
In this subsection we prove the following theorem.

\begin{thm}\label{thm:invariance}
Let $\pi: \tilde{X} \rightarrow X$ be the blow up of a smooth projective $3$-fold along a smooth subvariety $S$. Also let $C$ be a curve in $\tx$ which does not intersect the exceptional divisor $E$. Then there is a non-zero descendant Gromov-Witten invariant on $\tx$ of the form 
\[
\langle [C], \tau_{d_1}A_1, \ldots, \tau_{d_n} A_n \rangle^{\tx}_{0, \beta}
\]
if and only if there is a non-zero descendant Gromov-Witten invariant on $X$ of the form 
\[
\langle [C], \tau_{d_1}A_1, \ldots, \tau_{d_m} A_m \rangle^{X}_{0, \beta'}.
\]
Here we use $C$ to denote both the curve on $\tx$ and its image in $X$.
\end{thm}
\begin{proof}
Suppose there is a non-zero descendant Gromov-Witten invariant on $X$ of the form 
\[
\langle [C], \tau_{d_1}A_1, \ldots, \tau_{d_n} A_1 \rangle^{X}_{0, \beta'}.
\]
We may assume that all the $A_i$ are of the form $\sigma_i$. We degenerate $X$ into $\tx$ and $\PP_S(\OO \oplus N_{S/X})$ and apply the degeneration formula. So there is a non-zero relative invariant:
\[
\langle [C], \tau_{d_1}A_1, \ldots, \tau_{d_k} A_k \vert \mu \rangle^{\tx, E}_{0, \beta}.
\]
Then in the setup of Theorem~\ref{thm:Correspondence}, take $S$ to be $E$. Note that the blow up of $\tx$ with center $E$ is $\tx$ itself. So the theorem gives an absolute invariant of $\tx$ of the desired form.

Conversely, suppose there is a non-zero descendant Gromov-Witten invariant on $\tx$ of the form 
\[
\langle [C], \tau_{d_1}A_1, \ldots, \tau_{d_n} A_n \rangle^{\tx}_{0, \beta}.
\]
We may assume that $A_i, 1 \leq i \leq m$ are of the form $\pi^* \sigma_{i_j}$ and $A_i, m+1 \leq i \leq n$ are of the form $\iota_*(\delta_{i_j})$. Then we degenerate $\tx$ into $\tx$ and $\PP_E(\OO \oplus N_{E/X})$. We specialize $A_i, m+1 \leq i \leq n$ to the projective bundle side. Then we get a relative invariant of the form:
\[
\langle [C], \tau_{d_1}A_1, \ldots, \tau_{d_k} A_k \vert \mu \rangle^{\tx, E}_{0, \beta}.
\]
with $k \leq m$. In particular, all the $A_i, 1 \leq i \leq k$ are of the form $\pi^* \sigma_{i_j}$. Again apply Theorem~\ref{thm:Correspondence} to $(\tx, E)$ and $(X, S)$ we get an absolute descendant invariant of desired form.

\end{proof}

\subsection{Symplectic Invariance}

We first recall the proof of Theorem \ref{thm:KoRu} in ~\cite{KollarUni} and ~\cite{RuanUni}. 
\begin{proof}[Proof of Theorem \ref{thm:KoRu}]
We first choose a polarization of $X$. Then there exists a minimal free curve $C$ with respect to the polarization. Note that every rational curve through a very general point $p$ in $X$ is free. So if we choose such a point and consider all the curves mapping to $X$ of class $[C]$ and passing through $p$, then we get a proper family ( since none of them can break by minimality) of expected dimension ( the deformation is unobstructed). Therefore the Gromov-Witten invariant $\langle [pt], A^2, \ldots A^2\rangle_{0, [C]}^{X}$ is non-zero, where $A$ is a very ample divisor. Clearly this is the number of curves meeting all the constraints.

\end{proof}

\begin{proof}[Proof of symplectic invariance]
By Theorem~\ref{thm:KoRu}, $X$ and $X'$ are symplectic uniruled. Now we can look at the \emph{maximal rationally connected (MRC) quotient} of $X'$: a rational map $X' \dashrightarrow S$ such that the closure of a general fiber ( in $X'$) is the equivalence class of points in $X'$ that are connected by a chain of rational curves. $X'$ is rationally connected if and only if $S$ is point. In our case, $\dim S \leq 2$. By the theorem of Graber-Harris-Starr (\cite{GHS03}), $S$ is not uniruled. 

We will use proof by contradiction. So assume $S$ is not a point. If $S$ is a curve, it cannot be a rational curve. Then we get a non-zero section of $H^0(X', \Omega_{X'})$. But $X'$ is simply connected since $X$ and $X'$ are diffeomorphic and $X$ is simply connected. Then by Hodge decomposition, $H^0(X', \Omega_{X'})=0$. This is a contradiction. So $S$ has to be a non-uniruled surface and the general fiber of the rational map $X' \dashrightarrow S$ is a rational curve. 

Note that in this case every fiber class is the curve class $[C]$ by minimality. So we see that on $X'$ we have 
\[
\langle [pt] \rangle_{0, [C]}^{X'}=1, -K_{X'} \cdot C=2.
\]
 Both of these conditions are symplectic deformation invariant. So we get on $X$:
\[
\langle [pt] \rangle_{0, [C]}^{X}=1, -K_X \cdot C=2.
\]

We have seen that this Gromov-Witten invariant is enumerative. So there is only one minimal free curve passing through a general point in $X$. Let
\[
\pi: \mcC \rightarrow \Sigma
\]
be the universal family of the minimal free curves and 
\[
f: \mcC \rightarrow X
\]
be the universal map. Then $f$ is birational. Thus there is a rational map $X \dashrightarrow \Sigma$. That is, $X$ is birational to a conic bundle over a rational surface.

Let $\Gamma \subset X \times \Sigma$ be the closure of the rational map $X \dashrightarrow \Sigma$. Then it is easy to see that the exceptional divisors of $\Gamma \rightarrow X$ do not dominate $\Sigma$. So there is an open subset $U$ of $X$ and a smooth open subset $V$ of $\Sigma$ such that $U \rightarrow V$ is a well defined proper morphism and a general fiber is $\PP^1$. we can choose smooth projective compactifications of $U$ and $V$, denoted by $Y$ and $\Sigma{'}$, together with a morphism $Y\rightarrow \Sigma{'}$ such that a general fiber is $\PP^1$. By the weak factorization theorem in~\cite{AKMW}, there is a sequence of blow-up/ blow-downs
\[
X=X_0 \dashrightarrow X_1 \ldots \dashrightarrow X_n=Y
\]
such that every birational map is an isomorphism over $U$. In particular, there is a free curve $C$ in every $X_i$ away from every exceptional divisor.

By Theorem~\ref{thm:conic}, there is a non-zero Gromov-Witten invariant on $Y$ of the form $\langle [C], \ldots \rangle^Y_\beta$. So by Theorem~\ref{thm:invariance}, there is a descendant Gromov-Witten invariant on $X$ of the form $\langle [C], \tau_{d_1}A_1, \ldots \rangle^X_{\beta'}$. Hence also in $X'$. We only need to show that the curve $\beta'$ is not supported in a fiber of $X' \dashrightarrow S$. If so, then $\beta'$ is a multiple of $[C]$ and $-K_{X'} \cdot \beta' \geq 2$. So there are other insertions in the descendant invariant. But notice that by the proof of Theorem~\ref{thm:conic} and ~\ref{thm:invariance}, all such insertions are of the form $\tau_d A, A \in H^{\geq 4}(X', \QQ)$. So we can take representatives of the cycle $A$ away from a general fiber $C$. Then the invariant should be zero since a curve supported in a fiber cannot meet the cycle representing $A$. Therefore we have shown that $S$ is uniruled by the images of $\beta'$. This is a contradiction. So $X'$ is actually rationally connected.

\end{proof}

\section{Fano threefolds}\label{fano}
The main result in this section is the following theorem.

\begin{thm}\label{thm:Fano3}
If $X$ is a Fano threefold, then there is a non-zero Gromov-Witten invariant of the form $\langle [pt], [pt], [A_3], \ldots, [A_n]\rangle_{0, \beta}^{X}$. Moreover, this Gromov-Witten invariant is enumerative.
\end{thm}

\subsection{Some results from birational geometry}
In this section we collect some results on the classification of $K_X$-negative extremal contractions on a smooth projective variety. 
\begin{thm}[\cite{KM98}~\cite{kollar1994}]\label{thm:Mori}
Let $X$ be a smooth threefold and $contr: X \rightarrow Y$ be a contraction of $K_X$-negative extremal ray. Then one of the followings holds.
\begin{itemize}
\item{(E1)} $Y$ is smooth and $X$ is the blow up of $Y$ along a smooth curve.
\item{(E2)} $Y$ is smooth and $X$ is the blow up of $Y$ along a point.
\item{(E3)} $Y$ is singular and locally analytically isomorphic to $x^2+y^2+z^2+w^2=0$. $X$ is the blow up at the singular point.
\item{(E4)} $Y$ is singular and locally analytically isomorphic to $x^2+y^2+z^2+w^3=0$. $X$ is the blow up at the singular point.
\item{(E5)} $contr$ contracts a smooth $\PP^2$ with normal bundle $\OO_{\PP^2}(-2)$ to a point of multiplicity $4$ in $Y$.
\item{(C)} $Y$ is a smooth surface and $X$ is a conic bundle over $Y$.
\item{(D)} $Y$ is a smooth curve and $X$ is a fiberation in Del Pezzo surfaces.
\item{(F)} $X$ is a Fano $3$-fold with Picard number one and $Y$ is a point.
\end{itemize}
\end{thm}

It is very easy to workout the exceptional divisor in the exceptional contraction cases. And we have the following corollary.
\begin{cor}\label{cor:normalbundle}
In the case of (E2)-(E5), the exceptional divisor is rationally connected and the following is the list of minimal very free curves in the exceptional divisor and their normal bundles in $X$.
\begin{itemize}
\item{(E2)} A line $L$ in $\PP^2$, $N_{L/X}\cong \OO(1)\oplus \OO(-1)$.
\item{(E3)} A conic $C$ in a smooth quadric hypersurface. $N_{C/X}\cong \OO(2) \oplus \OO(-2)$.
\item{(E4)} A conic $C$ in a quadric cone. $N_{C/X}\cong \OO(2) \oplus \OO(-2)$.
\item{(E5)} A line $L$ in $\PP^2$, $N_{L/X}\cong \OO(1) \oplus \OO(-2)$.
\end{itemize}
\end{cor}

We also need the following result from \cite{MoMu2} and \cite{MoMu1}.
\begin{prop}\label{prop:blowup}
Let $X$ be a Fano threefold and $contr: X \rightarrow Y$ be the blow up along a smooth curve. Then $Y$ is Fano unless we are in the following case:
$X$ is the blow up along a smooth $\PP^1$ whose normal bundle in $Y$ is $\OO(-1)\oplus \OO(-1)$. 

In this case, the exceptional divisor $E$ is isomorphic to $\PP^1 \times \PP^1$ and the normal bundle of the curve of bi-degree $(1,1)$ is $\OO(2)\oplus \OO(-2)$.
\end{prop}

\begin{defn}
Let $X$ be a Fano $3$-fold. We say $X$ is primitive if it is not the blow up along a smooth curve of another Fano $3$-fold.
\end{defn}
\subsection{Construction of low degree very free curves}

We first show the following theorem.
\begin{thm}\label{thm:lowdegree}
Let $X$ be a Fano threefold. Then there is very free curve whose $-K_X$ degree is at most $6$.
\end{thm}

\begin{proof}
We may assume $X$ is a primitive Fano $3$-fold since the very free curve can be moved away from any blow up center and gives a very free curve in the blow up with the same anticanonical degree. Under this assumption, every exceptional divisor is rationally connected by Corollary~\ref{cor:normalbundle} and Proposition~\ref{prop:blowup}.

We first consider the case $\rho(X)=1$. 

Fix a polarization on $X$. Let $C$ be a minimal free rational curve. If $-K_X \cdot C=4$, then $N_{C/X}\cong \OO(1)\oplus \OO(1)$. We are done in this case. 

If $-K_X\cdot C =3$, then we can fix a general point in $X$ and the deformation of $C$ fixing that point sweeps out a surface $S$. Since $\rho(X)=1$, $S$ is ample. If we take a minimal free curve through another general point, $S$ intersect that curve in a finite number of points. So we have a reducilbe curve, which is the union of two free curves passing through two general points in $X$. So a general deformation of this curve is an irreducible curve passing through two general points, thus very free. The anti-canonical degree of the very free curve is $6$.

If $-K_X\cdot C =2$, then it is proved in Corollary 4.14 (Chapter IV) and Proposition 2.6 (Chapter V) of ~\cite{Kollar96} that there is a chain of free curves of length at most $3=(\dim X)$ connecting two general points. Thus a deformation of this chain gives a very free curve whose anti-canonical degree is at most $6 (=2 \dim X)$.

Next assume $\rho(X)\geq 2$. It is proved in ~\cite{MoMu1} that there is an extremal ray R1 corresponding to a contraction of type (C) (c.f. Theorem 7.1.6 in ~\cite{AG_Fano}). Let $C$ be a general fiber. 

The discussion is divided into three parts according to the type of contraction given by other extremal rays.

 If there is an extremal ray corresponding to a divisorial contraction with exceptional divisor $E$, then $E \cdot C > 0$.  In fact, we may find a possibly reducible curve of class $[C]$ which intersects $E$ since $C$ is free. But then $E \cdot C >0$ since $C$ spans an extremal ray. By Corollary ~\ref{cor:normalbundle} and Proposition~\ref{prop:blowup}, the exceptional divisor is rationally connected. So we can assemble a chain of rational curves in the following way. Take two minimal free curves each through a general point and a minimal very free curve in $E$ connecting the two minimal free curves. Then one can see that this chain of rational curves deforms to an irreducible rational curve by looking at the deformation and obstruction spaces of this reducible curve. A general deformation will again pass through $2$ general points, thus very free. The $-K_X$ degree is at most $6$.

 If there is an extremal ray of type (D), $X$ is a Del Pezzo fiberation over $\PP^1$. We can assemble a connected reducible curve by gluing a free curve $C$ to a minimal very free curve in a general fiber. It is easy to see that a general deformation of this curve is very free with $-K_X$ degree at most $6$.

Finally, assume all the other extremal rays are of type (C). Then there are at least two free curves $C_1$ and $C_2$ through a general point. A general deformation of the union of these two curves is a free curve of $-K_X$ degree $4$. If it is not very free, we may deform this curve fixing a point and get a divisor $H$. Note that $H$ is nef since we can move $H$ away from any curve by changing the point. We are done if $H \cdot C_i>0$ for $i=1$ or $2$. If $H\cdot C_i=0$, then $\rho(X) \geq 3$. Some multiple of $H$ is base point free and defines a morphism $X \rightarrow \PP^1$, which contracts both $C_1$ and $C_2$. 
There is a third extremal ray, of type $C$ by our assumption. Then we get a very free curve with anticanonical degree no more than $6$ by the same argument as above.

\end{proof}

\subsection{Fano threefolds are symplectic rationally connected}
\begin{proof}[Proof of Theorem~\ref{thm:Fano3}]
By Theorem~\ref{thm:lowdegree}, there is a very free curve with $-K_X$ degree no more than $6$. Take a minimal such curve $C$. We consider Gromov-Witten invariant of the form $\langle [pt], [pt], [A^2], \ldots, [A^2] \rangle^X_{0, C}$, where $A$ is the class of a very ample divisor. Note that in any case, the component whose general member is a very free curve contribute positively to this Gromov-Witten invariant. We will show that the contribution from every component of the moduli space is non-negative. Most of the time we do not distinguish the stable map and the curve it maps to.

We first choose two very general points such that any irreducible rational curve through them is very free and any rational curve through one of them is free. In particular, the anti-canonical degree of any curve through one of them is at least $2$. Note that in any degeneration, the two points are not connected by an irreducible curve by minimality of the very free curve. Also notice that if the degeneration is a union of free curves, then this degeneration is in the irreducible component whose general member is a free curve. Since the irreducible component is of expected dimension, the contribution to the Gromov-Witten invariant is non-negative. Thus we only need to consider the case where some components of the curve are not free.

 We discuss the degenerations according to the anti-canonical degree of $C$. 
\begin{enumerate}
\item $-K_X \cdot C=4$. Then the only possible degeneration is two free curves. We are done in this case.
\item $-K_X \cdot C=5$. Then we may add one more constraint, a curve which is the complete intersection of two very ample divisors. It suffices to consider the degeneration into $3$ irreducible components: two of them, $C_1$ and $C_2$, are free curves passing through the two chosen points together with one curve $D$ such that $-K_X \cdot C_1= -K_X \cdot C_2=2$ and $-K_X \cdot D=1$. Notice that $D$ only deforms in a surface $S$ and $C_1, C_2$ does not deform once the two points are fixed. Both $C_1$ and $C_2$ intersect the surface $S$ at finitely many points and $D$ has to pass through at least two of them. Then $D$ does not move once we make the choice of the two points. Otherwise we can deform $D$ fixing two points and, by bend-and-break, $D$ breaks into a reducible or non-reduced curve. But this cannot happen since $-K_X \cdot D=1$ and $-K_X$ is ample. So we see that there are only finitely many such degenerations. Now we can choose the curve representing the constraint in such a way that it misses these finitely many curves.
\item $-K_X \cdot C=6$. In this case we add two more constraints, both of which are curves coming from complete intersections of very ample divisors. There are more possible degenerations.

We observe that no degree $1$ curve can connect  two free curves of $-K_X$ degree $2$ and passing through a general point. Otherwise there is a chain of rational curves connecting two general points with anti-canonical degree $5$. We claim that a general deformation of this reducible curve smooths the nodes. If not, then a general deformation is given by the deformation of the two free curves and the deformation of the degree $1$ curve. Let $S$ be the surface swept out by the degree $1$ curve. Notice that up to finitely many choices, the deformation of the degree $1$ curve is determined by the intersection of the two free curves with the surface $S$: fix two points in the surface, there are only finitely many choices for the degree $1$ curve. So the deformation space has dimension $2+2=4$. But every irreducible component containing this reducible curve has dimension at least $5$(=$-K_X$ degree). So a general deformation smooths at least one of the nodes and we get two irreducible curves each passing through a general point. Then it is easy to see that a general deformation of this curve is very free. So we have constructed a very free curve with lower $-K_X$ degree, a contradiction to our assumption.

First consider the degeneration with four irreducible components whose $-K_X$ degrees are $2, 2, 1, 1$. 

If the two curves of degree $1$ deform in two different surfaces, then each of the degree $1$ curve will connect one (and only one by the above observation) free curve and the other degree $1$ curve. We first choose our constraint curve to avoid the degree $2$ curves. Then one of the degree one curve has to pass through the intersection point of this constraint curve with the surface swept out by itself. Notice that there are only finitely many such choices for the same reason as before. Also once we make the choice for one curve, there are only finitely many choices for the other degree $1$ curve since it has to pass through the intersection of the degree $2$ and the degree $1$ curve with the surface. So we can choose the last constrain to avoid all of these curves. 

Now assume the two degree one curves deform in the same surface. We may choose the constrain curve to intersect the surface at general points. If they meet all the constraints, then one of the degree $1$ curve has to pass through both the intersection of the free curve with the surface and a general point in the surface (i.e. the intersection of one constraint curve with the surface). Thus we can deform it in the surface and connect the two free curves by it. This is impossible by the previous observation.

Next, we consider the degeneration with three irreducible components whose $-K_X$ degrees are $2, 3, 1$. Notice that the free curve with anti-canonical degree $3$ is fixed once we fix a general point and a curve it passes through. Also the two free curves cannot meet each other by minimality. Then basically the same argument as before shows that we can choose the constraints to miss this configuration.

The last possible degeneration is the case where we have $3$ irreducible components with $-K_X$ degree $2$ and one of them is not free. If the non-free curve is the specialization of a free curve, then we can choose the constraints to miss such configurations. So we may assume the non-free curve only deforms in a surface. If this reducible curve can meet all the constraints, then the only non-free curve passes through the intersections of the free curves with the surface and two other general points in the surface. So after we fix the two intersection points and a general point, the curve deforms in a positive dimensional family. Then it breaks into two irreducible components or a non-reduced curve. In this way we get a rational curve with $-K_X$ degree $1$ and passing through one intersection point and a general point in the surface. So the same argument as before shows that we have a very free curve whose $-K_X$ degree is $5$. 

\end{enumerate}

\end{proof}

\begin{rem}
For another proof, see Section~\ref{sec:QFano}.
\end{rem}

\subsection{Some corollaries}
\begin{cor}\label{cor:fanoblowup}
Let $X$ be a rationally connected variety which is obtained by successive blow ups from a Fano threefold $Y$. Then $X$ is symplectic rationally connected.
\end{cor}

\begin{proof}
The minimal very free curve in $Y$ can be deformed away from the blow up centers. We can also choose the constraints in Theorem ~\ref{thm:Fano3} to be away from the centers. Then the very free curves in $Y$ meeting all the constraints are all away from the blow up centers. Observe that the image of any curve satisfying the constraints in $X$ also satisfies the constraints in $Y$. Thus the images are just the minimal very free curves. Then it follows that no components are contracted by the map $X \rightarrow Y$ and the curves in $X$ that can meet these constraints are again very free curves.

\end{proof}

\begin{rem}
As far as the author is aware of, there is no general blow up formulas for Gromov-Witten invariants. All the blow up formulas assume some positivity of the normal bundle of the blow up center. 
\end{rem}


\begin{cor}
Let $X$ be a Fano threefold and $Y\rightarrow X$ a $\PP^1$ bundle over $X$. Then $Y$ is symplectic rationally connected.
\end{cor}

\begin{proof}
First we choose constraints in $X$ that give the enumerative non-zero Gromov-Witten invariant in Theorem~\ref{thm:Fano3}. So there are only finitely many embedded very free curves satisfying these constraints. Note that $Y \rightarrow X$ is a smooth morphism, thus the inverse images of these curves are smooth ruled surfaces over $\PP^1$. We may take a section $\Gamma$. Let $F$ be the class of a fiber. Now take the constrains in $Y$ to be two points that are mapped to the two points in $X$ and the pull back of other constrains in $X$. Choose $k$ to be large enough. Then any curve in class $\Gamma+k F$ satisfying these constraints are in these ruled surfaces. Thus we may add other constraints to be surfaces intersecting these ruled surfaces at finitely many points. Then we are basically reduced to show that there are non-zero Gromov-Witten invariants 
\[
\langle [pt], \ldots, [pt] \rangle
\]
on a ruled surface. This is easy and follows from Proposition~\ref{prop:surface}.

\end{proof}

\begin{rem}
Let $X$ be a Fano $4$-fold of pseudo-index $2$ and $\rho(X) \geq 2$. If $X$ has a divisorial contraction that contracts a divisor to a point, or a fiber type contraction to a curve, then $X$ is a $\PP^1$ bundle over a smooth Fano $3$-fold. In ~\cite{AG_Fano}, this is proved under the assumption that the index is $2$. But the proof is the same if we only assume the pseudo-index is $2$.
\end{rem}

\section{Del Pezzo fiberations}\label{sec:delpezzo}

Let $X$ be a normal projective threefold with at worst terminal singularities. And let $\pi: X\rightarrow \PP^1$ be a contraction of $K_X$-negative extremal face. Then a general fiber is a smooth Del Pezzo surface. Let $f: Y \rightarrow X$ be a resolution of singularity that is isomorphic near a general fiber. Note that all the exceptional divisors are supported in some fibers of $\pi \circ f: Y \rightarrow X \rightarrow \PP^1$. The main result in this section is the following.

\begin{thm}\label{thm:symGHS}
There is a non-zero Gromov-Witten invariant $\langle [pt], [pt], \ldots\rangle_{0, \beta}^{Y}$ for some class $\beta$ which is a section of the fiberation.
\end{thm}

\begin{proof}
Fix a polarization on $Y$. By definition of terminal singularity, we have $K_Y=f^* K_X+\sum a_i E_i, a_i >0$. 

By the main theorem in ~\cite{GHS03}, there exists a section of $Y \rightarrow \PP^1$. 

 Let 
\[
A=min\{d \mid (\sum a_i E_i)\cdot s=d, s ~\text{is a section}\}.
\]
 Once we have a section, we can always attach very free curves in general fibers and deform the reducible curve to get a section which is free. This operation will not change the intersection number with the exceptional divisors as long as we attach very free curves in general fibers not containing the exceptional divisors. So there is a free section $s$ such that $(\sum a_i E_i)\cdot s=A$.

Define
\[
B=min\{b\geq 0\mid s \text{ is a section},
             s \cdot (\sum a_i E_i)=A, N_{s/Y}\cong \OO(a)\oplus\OO(a+b), a, b \geq 0\}.
\]

Let $s$ be a free section such that $(\sum a_i E_i) \cdot s =A$, $N_{s/Y}\cong \OO(a)\oplus \OO(a+B), a, B \geq 0$. 

A general fiber of $Y \rightarrow \PP^1$ is a Del Pezzo surface. So it is either $\PP^1 \times \PP^1$, $\PP^2$ or $\PP^2$ blow up at $d (1\leq d \leq 8)$ points. 

If a general fiber is not $\PP^1 \times \PP^1$, then the normal bundle of the section $s$ is $\OO(a) \oplus \OO(a)$. If not, assume it is $\OO(a)\oplus \OO(a+b), b>0$. we attach a line $L$ in a general fiber to $s$ along a general direction. Let $\mcN$ be the normal bundle of this reducible curve in $Y$. Choose a point $p$ in the line $L$ and a divisor $D=q_1+\ldots +q_{a+2}$ in $s$. Let $\mcE=\mcN(-p-D)$. Then 
\[
\mcE \vert_L \cong \OO\oplus \OO, \mcE \vert_s \cong \OO(-1) \oplus \OO(b-2).
\]
We have the short exact sequence of sheaves
\[
0 \to \mcE\vert_L (-n) \to \mcE \to \mcE \vert_s \to 0,
\]
where $n$ in the node of $L \cup s$.
Then $H^1( \mcE)=0$. The same is true for a general deformation by semi-continuity. Thus a general deformation is again a free section $s'$ with $N_{s'/Y}\cong \OO(a')\oplus \OO(a'+b'), a' \geq a+2, b' < b$. This also shows that there are free sections with normal bundle $\OO(a) \oplus \OO(a)$, $a\gg 0$. 

If a general fiber is $\PP^1\times \PP^1$, then the situation is slightly different. If $B>0$, then by the same argument above, we can find very free sections whose normal bundle is $\OO(a)\oplus \OO(a+B)$ with $a$ arbitrarily large by attaching curves of bi-degree (1, 1) in a general fiber to $s$ along a general normal direction. The case that a general fiber is $\PP^1\times \PP^1$ and $B=0$ will be discussed later.

There exist an $\epsilon >0$ such that $\pi ^*\OO(1)-\epsilon K_X$ is ample on $X$ by the construction of $K_X$-extremal contraction. So there are $b_i$'s such that
\begin{align*}
H=&f^*(\pi ^*\OO(1)-\epsilon K_X)-\sum b_i E_i\\
=&f^*\pi ^*\OO(1)-\epsilon K_Y+\sum (\epsilon a_i-b_i) E_i
\end{align*}
is ample. 

Let $s$ be a section with $(-K_Y)\cdot s \leq B$. Then $H \cdot s \leq 1+\epsilon  B+\sum \vert \epsilon a_i-b_i \vert$ since $E_i \cdot s=0$ or $1$. So there are only finitely many such curve classes. Then there is an integer $M$ such that any such sections can meet at most $M$ general curves and there are only finitely many such sections can meet $M$ general curves. We also have a lower bound $-K_Y \cdot s \geq N$ for all such sections.

The above discussion shows that there is a very free section $s$ such that 
\[
(\sum a_i E_i) \cdot s =A,
\]
\[
N_{s/Y} \cong \OO(a)\oplus \OO(a+B),
\]
\[
2M+3(a+1-M)+N >2+2a+B.
\]

Fix one such $a$. Denote by $s$ the minimal section with respect to some polarization in the set of all the very free sections with the above properties. We will prove that this curve class $[s]$ gives a non-zero Gromov-Witten invariant with two point insertions.

A general such section passes through $a+1$ general points. If $B>0$, then we may also add constraint curves in general fibers. Then a general section will meet all of these constraints and contribute positively to the Gromov-Witten invariant $\langle [pt], \ldots, [pt], [\text{Curve}], \ldots, [\text{Curve}]\rangle_{0, s}^{X}$.

Now we will show that any reducible curve in this curve class cannot meet the all the constraints. Write the reducible curve as $C \cup C_e \cup C_g$ where $C$ is a section, $C_e$ the vertical components supported in the fibers where $f$ is not an isomorphism, and $C_g$ all the other vertical components. 

First notice that $E_i \cdot C_g=0$. So $(\sum a_i E_i) \cdot s = (\sum a_i E_i) \cdot (C_e+C)$. Then $(\sum a_i E_i) \cdot C_e\leq 0$ since $(\sum a_i E_i) \cdot s \leq (\sum a_i E_i) \cdot C$. Therefore 
\[
-K_Y \cdot C_e=-f^*K_X \cdot C_e -(\sum a_i E_i) \cdot C_e \geq 0
\]
with equality if and only if $(\sum a_i E_i) \cdot C =A$ and $C_e$ is mapped to a point in $X$. There are three different cases according to what kind of curve $C$ is.
\begin{enumerate}
\item $C$ is a free curve. 

If $-K_Y \cdot C_e >0$, then $-K_Y \cdot (C+C_g) < 2+2a+B$. Suppose $C$ meets $a'$ general points and $b' ( \leq B)$ general curves in the fiber. Then $-K_Y \cdot C \geq 2a'+b' $. For the remaining constraints, a curve through a general point has $-K_Y$ degree at least $2$. A curve in a general fiber has $-K_Y$ degree at least $1$. Then the total $-K_Y$ degree of $C_e+C+C_g$ is strictly greater than $2a'+b'+2(a+1-a')+(B-b')=2+2a+B$. This is impossible. 

So $-K_Y \cdot C_e=0$ and $(\sum a_i E_i)\cdot C=A$. Let $N_{C/Y}\cong \OO(a'')\oplus \OO(a''+b'')$. Then $b'' \geq B$ by definition. Note that $C_e$ does not meet any constraints since we choose all the constraints to lie in general fibers. Also note that $C$ can pass through at most $a''+1$ points. A similar argument as above shows that the $-K_Y$ degree does not match unless $b''=B$, $C$ meet $a''+1$ points together with $B$ general curves, and $C_g$ is a bunch of free curves whose $-K_Y$ degree is $2$. Then the reducible curve $C \cup C_g$ deforms to an irreducible section curve which may pass $a+1$ general points and $B$ general curves in general fibers. Notice that this forces the normal bundle of the new section curve to be $\OO(a)\oplus (a+B)$. Then by the minimality of the section curve we start with, $C_e$ is zero. In other words, $C\cup C_e \cup C_g=C \cup C_g$ is in the boundary of the component of expected dimension. So we can choose the constraints to miss them.

\item $C$ is not a free curve and $-K_Y \cdot C > B$. We may choose the $a+1$ points lie in different general fibers and any curve through them is free. Then neither $C$ or $C_e$ passes through any of them. So $-K_Y \cdot C_g \geq 2(a+1)$ and $-K_Y \cdot (C_e+C + C_g) > 0 + 2a+2+B$. This is impossible. 

\item $C$ is not free and $-K_Y \cdot C \leq B$. Again $C$ does not meet any point constraints. So $C_g$ has at least $a+1$ curves $D_i$ in different fibers and $-K_Y \cdot D_i \geq 2$. If $-K_Y \cdot D_i =2$, $D_i$ is an irreducible free curve. There can be at most $M$ such curves and only finitely many sections can meet all of these curves. So if we choose other points to be general, then every curve through those points with $-K_X$ degree $2$ will not meet these sections. Thus for all the other $D_i$'s ( which are possibly reducible), we have $-K_Y \cdot D_i \geq 3$. But again the total $-K_Y$ degree of $C_e+C+C_g$ is greater than $-K_Y \cdot s$ by our choice of $a$.

\end{enumerate}

Finally we discuss the case where a general fiber is $\PP^1\times \PP^1$ and $B=0$. If there are infinitely many sections with normal bundle $\OO(a) \oplus \OO(a)$, then the above argument still works since we can choose $a$ to be large enough. 

So the only remaining case is that there are only finitely many such sections. In particular, there is an upper bound $\alpha$ for all such $a$. Let
\[
B'=min\{b> 0\mid s \text{ is a section},
             s \cdot (\sum a_i E_i)=A, N_{s/Y}\cong \OO(a)\oplus\OO(a+b), a> 0\}.
\]
A similar construction shows that there are sections with the following properties.
\[
(\sum a_i E_i) \cdot s =A,
\]
\[
N_{s/Y} \cong \OO(a)\oplus \OO(a+B'),
\]
\[
2M+3(a+1-M)+N >2+2a+B', a > \alpha.
\]
Again choose the minimal such section with respect to the polarization on $Y$. Then we only need to add one more case of the above proof: $C$ is a free section and $(\sum a_i E_i) \cdot C=A$, $N_{C/Y}=\OO(a') \oplus \OO(a')$. But any irreducible component of $C_g$ which lies in a smooth fiber is free since $\PP^1 \times \PP^1$ is convex. We may add constraints only in smooth fibers. So there is a sub-curve of $C\cup C_g$ which deforms to an irreducible free curve of $-K_Y$ degree at most $2+2a+B'$. Then the normal bundle of the new section has to be $\OO(a)\oplus \OO(a+B')$ since it also meet all the constraints and $a > \alpha$. So $C_e\cup C \cup C_g=C\cup C_g$ and every irreducible component of $C_g$ lies in some smooth fiber. Thus the corresponding point in the moduli lies in the boundary of an irreducible component whose general point is a free curve. We can choose the constraints to avoid such configurations.

\end{proof}

\section{Conic bundles}\label{sec:conic}

Our main goal in this section is to prove the following theorem. 

\begin{thm}\label{thm:conic}
Let $X \rightarrow \Sigma$ be a surjective morphism from a smooth projective rationally connected $3$-fold $X$ to a smooth projective surface $\Sigma$ such that a general fiber is isomorphic to $\PP^1$. Then $X$ is symplectic rationally connected. There is also a non-zero Gromov-Witten invariant of the form $\langle [C], \ldots\rangle_{0, \beta}^{X}$, where $[C]$ is the class of a general fiber.
\end{thm}

There are easier ways to prove the following result. But here we present a proof which only depends on MMP on surfaces and requires nothing about the classification of surface. This proof actually motivates the results in Section~\ref{sec:delpezzo}.
\begin{prop}\label{prop:surface}
Let $\Sigma$ be a rationally connected surface. Then there is a non-zero Gromov-Witten invariant of the form $\langle [pt], [pt], \ldots, [pt]\rangle_{0, \beta}^{\Sigma }$. And this invariant is enumerative.
\end{prop}

\begin{proof}

We can run the Minimal model program for $\Sigma$. Then we have a sequence of contractions of $(-1)$-curves: 
\[
\Sigma=X_0 \rightarrow X_1 \rightarrow \ldots \rightarrow X_n,
\]
where $X_n$ is either a geometrically ruled surface over $\PP^1$ or a Fano surface of Picard number $1$ (Thus is $\PP^2$, but we do not need this).

In the former case, we may choose a section $s_0$ of the ruled surface and take the curve class to be $s_0+k F$, where $F$ is a fiber class. If we take $k$ large enough, a general curve in this class is an embedded very free curve passing through $m( \geq 3)$ general points. Now one can use a similar argument as Theorem~\ref{thm:symGHS} to see that this section curve gives a non-zero enumerative Gromov-Witten invariant on $X_n$. 

In the latter case, first choose a minimal free rational curve. If it already very free, then we are done. If not, then notice that we can take the union of two such general curves and a general deformation is a very free curve of $-K_X$ degree $4$. It is easy to see that we get a non-zero enumerative Gromov-Witten invariant in this case. We remark here that by the classification of Del Pezzo surfaces, the minimal free curve is actually very free. But we do not need to know that.

Then the proposition follows from the observation that the corresponding curve class on the blow up gives a non-zero Gromov-Witten invariant.
\end{proof}

The proof of Theorem ~\ref{thm:conic} basically follows the same line as Theorem $2.4$ in~\cite{VoisinRC}, although the setup and statements are slightly different. We only point out necessary changes.
\begin{proof}[Proof of Theorem~\ref{thm:conic}]
By Proposition ~\ref{prop:surface}, there is a non-zero enumerative Gromov-Witten invariant of the form  
\[
\langle \underbrace{[pt], [pt], \ldots, [pt]}_{r ~{ [pt]}}\rangle_{0, \beta}^{\Sigma}.
\]
We also note that a general curve of class $[\beta]$ is an embedded curve. Here we need to know that a Del Pezzo surface of Picard number one is $\PP^2$. The curve class corresponds to a free linear system. So we may choose the constraints in $\Sigma$ to be general such that if $\Gamma$ is the curve through these points, then $Z=\pi^{-1}(\Gamma)$ is a smooth surface. Let $s_0$ be a section of $Z \rightarrow \Gamma$. Choose $k$ large enough. Then it is easy to see that any curve in $X$ in the class $s_0+k C$ which meets $2$ general points and $r-2$ general fibers (or $r$ general fiber $C$) has to be mapped to an irreducible curve in $\Sigma$ through $r$ general points. Thus the curve lies in $Z$. We may take other constraints to be curves meeting the surface $Z$ at finitely many points. There are some positive contributions to the Gromov-Witten invariant coming from these irreducible section curves. Now the problem is, there might be some curve in $Z$ whose curve class is not $\Gamma+k C$ in $Z$ but equals to $\Gamma+k C$ when considered as curve class in $X$. So we have to consider the contributions coming from these curves as well. This has been done in ~\cite{VoisinRC}. By deforming $Z$ to be the blow up of some Hirzebruch surface at \emph{distinct} points, it is shown there that all such contributions are non-negative.

\end{proof}

\begin{rem}
Let $Y \to \Sigma$ be a conic bundle coming from the contraction of a $K_Y$-negative extremal ray. Then we can take a resolution $\Sigma' \to \Sigma$ and a base change $Y' \to \Sigma'$. After further resolving the singularities of $Y'$, we end up with a conic bundle $\tilde{Y} \rightarrow \Sigma'$ such that $\tilde{Y}$ is smooth. This is how one prove the first case of Theorem~\ref{thm:SRC}.
\end{rem}

\begin{rem}
If $Y$ is a smooth $3$-fold and $Y \rightarrow \Sigma$ is a contraction of type (C), then $\Sigma$ is smooth. Also notice that in this case we do not even need to deform $Z$ in the proof of Theorem~\ref{thm:conic} since it is already the blow up of some Hirzebruch surface at \emph{distinct} point.
\end{rem}

\section{$Q$-Fano variety}\label{sec:QFano}
\begin{defn}
Let $X$ be a normal projective variety. $X$ is called a $Q$-Fano variety if $X$ is $Q$-factorial, has terminal singularities and $-K_X$ is ample.
\end{defn}

\begin{defn}
Let $f: C \rightarrow X^{\text{sm}}$ be an embedded very free curve and $N_{C/X} \cong \oplus_{i=1}^{n} \OO(a_i)$ such that $1 \leq a_1 \leq a_2 \leq \ldots, \leq a_n$. Then $C$ is a \emph{balanced curve} if $a_n-a_1 \leq 1$.
\end{defn}

\begin{lem}\label{lem:bound}
Let $X$ be a projective variety of dimension $n$ with terminal singularities.  Then for any rational curve $C$ passing through $r$ very general points , $-K_X \cdot C \geq (n-1)(r-1)+2$. And if equality holds, then $C$ is a curve contained in the smooth locus of $X$.
\end{lem}
\begin{proof}
Let $f: Y\rightarrow X$ be a resolution of singularity such that $X$ and $Y$ are isomorphic over the smooth locus of $X$. We have $K_Y = f^* K_X + \sum b_i E_i$ where $E_i$'s are exceptional divisors of $f$ and $b_i>0$. Then the normal bundle of an irreducible rational curve $C$ through $r$ very general points is 
\[
N_{C/Y}\cong \oplus_i \OO(a_i), a_i \geq r-1.
\]
So $-K_X \cdot C \geq -K_Y \cdot C \geq (n-1)(r-1)+2$. The last statement follows from the fact that $ E_i \cdot C \geq 0$ and equality holds if and only if $C$ does not meet $E_i$ or equivalently, $f(C)$ is contained in the smooth locus of $X$.
\end{proof}

\begin{defn}
Let $C_i \subset X_i$ be a curve on a variety $X_i$, $i=1, 2$. We say $(X_1,C_1)$ is \emph{equivalent to} $(X_2, C_2)$ if there is an open neighborhood $V_i$ of $C_i$ in $X_i$ and an isomorphism $f: V_1 \rightarrow V_2$ such that $f\vert_{C_1}:C_1 \rightarrow C_2$ is also an isomorphism.
\end{defn}

\begin{thm}\label{thm:balanced}
Let $X$ be a $Q$-Fano $3$-fold. Assume the smooth locus of $X$ is rationally connected, then there is a very free curve in the smooth locus with normal bundle $\OO(a) \oplus \OO(a), a \geq 1$.
\end{thm}

\begin{proof}
Let $f: Y \rightarrow X$ be a resolution of singularity which is isomorphic over the smooth locus of $X$. Let $C$ be a very free curve in the smooth locus and general in an irreducible component of moduli space of very free curve. We may assume that $-K_X \cdot C$ is an even number ( otherwise take a two-fold cover and a general deformation). Assume the normal bundle of $C$ is $\OO(a+2b) \oplus \OO(a), a, b \geq 1$. In the following we will not distinguish between $-K_X$ and $-f^*{K_X}$.

We can deform $C$ with $a+1$ points fixed and the deformation of $C$ sweeps out a surface $\Sigma$ in $Y$. Let $\Sigma'$ be the normalization and $\tilde{\Sigma}$ be a desingularization of $\Sigma'$. Then it is proved in ~\cite{ShenMM} that the pair $(\tilde{\Sigma}, C)$ is equivalent to $(\PP^2, conic)$ or $(\FF_n, \text{positive section})$ where $\FF_n$ is the $n$-th Hirzebruch surface. 

 Note that in the $\PP^2$ case $N_{C/Y}\cong \OO(4)\oplus \OO(2)$. So $-K_X \cdot C=-K_Y \cdot C=8$. $C$ may degenerate into two "lines" which passes through $2$ very general points, thus very free with normal bundle $\OO(1) \oplus \OO(1)$. Note that a general line is necessarily contained in the smooth locus since the intersection number with $-K_X$ is $4$, the same as the intersection number with $-K_Y$. So we are done.

Now assume that the pair is equivalent to $(\FF_n, \sigma)$, where $\sigma$ is a positive section. There is a (reducible) curve $D$ in $\tilde{\Sigma}$ such that $D^2=-n$ and $D \cdot F=1$. Then $C=D+cF$. It is easy to see that $N_{C/\tilde{\Sigma}}\cong \OO(a+2b)$ and $N_{\tilde{\Sigma/Y}}\vert_C \cong \OO(a)$. Thus,
\[
C^2=2c-n=a+2b,
\]
\[
c (-K_X \cdot F) \leq 2+2a+2b.
\]

Note that $-K_X \cdot F \geq 2$ since a general $F$ passes through a very general point. So $n \leq c \leq 1+a+b$. If $c=1+a+b$, then $-K_X \cdot F=2$ and $-K_X \cdot D=0$. The first equality implies that a general $F$ is a free curve in the smooth locus. The second one implies that $D$ is either mapped to a point or into the exceptional divisors of $f: Y \rightarrow X$. But $F \cdot D=1$ in $\tilde{\Sigma}$. So $D$ is mapped to a point in the smooth locus. This implies that there are two free curves in the smooth locus of $X$, each having $-K_X$ degree $2$ and passing through a very general point, meet in the smooth locus. But if we choose the two points to be general enough, any curve of $-K_X$ degree $2$ through these points will not meet each other. Therefore $c \leq a+b$ and $c-n \geq b$. 

Note that 
\[
-K_Y \cdot F\leq -K_X \cdot F \leq \frac{2+2a+2b}{c} = \frac{2(2+2a+2b)}{a+2b+n} \leq \frac{4(a+b+1)}{a+b+1}=4.
\]
Thus $-K_Y \cdot F=2, 3$, or $4$. 

If $-K_Y \cdot F=4$, then $n=0, b=1, -K_Y \cdot F= -K_X \cdot F, -K_X \cdot D=0$. So $F$ is a free curve in the smooth locus of $X$ and $D$ is contracted to a point in the smooth locus. But $(\tilde{\Sigma}, \sigma) $ is equivalent to $\PP^1 \times \PP^1$ with one ruling. So $D$ is a moving curve and $-K_X \cdot D >0$. This is a contradiction. 

It is proved in ~\cite{ShenMM} that there is a neighborhood $U$ of $C$ such that the map $\phi : \tilde{\Sigma} \rightarrow X$ has injective tangent map. Since $C \cdot D >0$, a general fiber $F$ is contained in $U$. Lemma $2.3.13$ in ~\cite{ShenMM} shows that for a general fiber $F$ in $U$, $N_{F/Y}$ cannot be $\OO\oplus \OO(1)$ ( assuming $C \cdot D >0$). Thus $-K_X \cdot F \geq -K_Y \cdot F=2$ and $N_{F/Y}=\OO\oplus \OO$.  

Note that $C$ may specialize to the union of a positive section $C'$ whose curve class is $D+(c-b)F$ with $b$ general fibers. Also note that $C'$ passes through $a+1$ general points in $\tilde{\Sigma}$ since its normal bundle in $\tilde{\Sigma}$ is $\OO(a)$. Since $C$ passes through $a+1$ general points in $X$, the same is true for $C'$. But $-K_X \cdot C' = -K_X \cdot C - b(-K_X) \cdot F \leq 2+2a$. Therefore the equality has to hold by Lemma~\ref{lem:bound} and $C'$ is a very free curve in the smooth locus with normal bundle $\OO(a)\oplus \OO(a)$.
\end{proof}

As a immediate corollary of this theorem and Theorem~\ref{thm:lowdegree}, we get
\begin{cor}\label{cor:FanoBalanced}
The minimal very free curve on a smooth Fano $3$-fold is balanced. And on every smooth Fano $3$-fold, there is an embedded very free curve with normal bundle $\OO(a)\oplus \OO(a), a\geq 1$.
\end{cor}

\begin{rem}
If $X$ is a $4$-fold with a contraction of $K_X$-negative extremal face to $\PP^1$, one may try to apply the same strategy in Sec. ~\ref{sec:delpezzo} to show that $X$ is symplectic rationally connected. Then one of the steps is to show that we can increase the $-K_X$ degree of a free curve without making it "less balanced". This corollary might be helpful.
\end{rem}

The importance of these balanced very free curves are clear from the following observation.

\begin{prop}\label{prop:QFano}
Let $X$ be $Q$-Fano $3$-fold and $f: Y \rightarrow X$ be a resolution of singularities. Assume there is a very free curve in the smooth locus of $X$ whose normal bundle is $\OO(a) \oplus \OO(a), a \geq 1$. Then $Y$ is symplectic rationally connected.
\end{prop}

\begin{proof}
Let $C$ be such a very free curve in the smooth locus of $X$. We can move $C$ away from the locus where $f$ is not an isomorphism. So we get a balanced very free curve $C$ in $Y$ with normal bundle $\OO(a) \oplus \OO(a), a \geq 1$. We choose the constraints to be $a+1$ general points. Then the result follows from the fact that any irreducible curve through $a$ general points has $-K_X$ degree at least $2a$ (c.f. Lemma~\ref{lem:bound}) and equality holds if and only if its image in $X$ is contained in the smooth locus.

\end{proof}

In particular, this gives a new proof that every smooth Fano $3$-fold is symplectic rationally connected. It is proved in ~\cite{ShenMM} that a general very free curve ( in each irreducible component of moduli space of very free curves) in a Fano $3$-fold of Picard number $1$ is balanced except the case that the $3$-fold is $\PP^3$ and the curve is a conic. Then we see that there are infinitely many non-zero Gromov-Witten invariants with two point insertions on such Fano $3$-folds.

In general it is not an easy task to determine if the smooth locus of a $Q$-Fano variety is rationally connected. In the following we prove that this is true for a large class of $Q$-Fano varieties we are interested in.

\begin{prop}\label{prop:Gorenstein}
Let $X$ be a Gorenstein $Q$-Fano 3-fold. Then the smooth locus of $X$ is rationally connected. 
\end{prop}
\begin{proof}
 By a result of Namikawa~\cite{Namikawa}, there is a smoothing of $X$, $\pi: \mcX \rightarrow S$ such that a general fiber is a smooth Fano variety and the central fiber $X_0$ is $X$.  

By Corollary ~\ref{cor:FanoBalanced}, there is a very free curve $D$ in a general fiber whose relative normal bundle is $\OO(a)\oplus \OO(a), a \geq 1$. So this curve will passing through $a+1$ general points in a general fiber. Now choose $a+1$ general points in $X$. We can find $a+1$ sections passing through these points in $X$, possibly after a base change. Then if we consider the specialization of the curve $D$ passing through these sections in the relative Kontsevich moduli space, we get a stable map to $X$ whose image contains $a+1$ general points. But as observed in ~\ref{prop:QFano}, the domain has to be irreducible and the image is contained in the smooth locus if the points are chosen to be general.

\end{proof}

\begin{rem}
It follows from the classification of $3$-fold terminal singularities that a Gorenstein terminal singularity is an isolated hypersurface singularity. Thus $X$ is l.c.i. Then the proposition can be proved easily by comparing the deformation space of a very free curve in a resolution and that of its image in $X$.
\end{rem}

\section{Varieties with $b_2(X)=2$}

In this section we prove that a smooth projective rationally connected $3$-fold $X$ with $b_2(X)=2$ is symplectic rationally connected.

There is at least a $K_X$-negative extremal ray. Let $f: X \rightarrow Y$ be the corresponding contraction. Then previous results ~\ref{thm:conic}, \ref{thm:symGHS}, \ref{cor:fanoblowup}, \ref{thm:balanced}, \ref{prop:QFano}, \ref{prop:Gorenstein} cover the cases of type (E1)-(E4), (C), (D). The only remaining case is (E5)-type contraction, where the exceptional divisor $E$ is a smooth $\PP^2$ with normal bundle $\OO(-2)$. $Y$ is a non-Gorenstein $Q$-Fano $3$-fold of Picard number $1$.

Let $C$ be a minimal free curve with respect to $-f^*K_Y$ and an ample divisor $A=-f^*K_Y -\epsilon E, 0 < \epsilon \ll 1$. If $E \cdot C=0$, then $C$ give rise to a free curve in the smooth locus of $Y$. Thus the smooth locus of $Y$ is rationally connected and the result follows from ~\ref{prop:QFano}. If $E \cdot C >0$, we first show that $-K_X \cdot C =2$. If $-K_X \cdot C\geq 3$, then we can deform a general such curve with one point fixed. The deformation in $X$ also gives a deformation of its image in $Y$. Since $E \cdot C \geq 1$, the deformation in $Y$ fixes the chosen point and the unique singular point in $Y$. So by bend-and-break, we get a curve through the same point with smaller $-K_Y$ degree. If the fixed point is chosen to be a very general point, then this curve with smaller $-K_Y$ degree gives a free curve in $X$. This is a contradiction to our choice.

We can take a reducible curve $\Gamma$ to be a union of two such free curves passing through two very general points and a line in $E\cong \PP^2$. It is easy to see that we can smooth the nodes of $\Gamma$ and get a very free curve $\Gamma'$ with $-K_X \cdot \Gamma'=5$. So we choose the constrains to be two very general points and a moving curve which does not meet $E$. 

If there is a degeneration such that the two points are connected by an irreducible curve $C_1$ and $-K_X \cdot C_1\geq 5$, then we can deform $C_1$ with two points fixed. So $C_1$ will degenerate to some reducible or non-reduced curves. If the two points are still connected by an irreducible curve $C_2$ with $-K_X \cdot C_2 \geq 5$, then we can do this again until we end up with one of the following.
\begin{enumerate}
\item There is an irreducible curve $C_n$ passing through the two very general points (thus very free) such that $-K_X \cdot C_n=4$.
\item There are two irreduible curves $D_1$ and $D_2$, each passing through one point.
\end{enumerate}

If the first case actually happens for some degeneration, we can choose a very free curve $D$ such that $-K_X \cdot D=4$ and its $-f^*K_Y$ degree is minimal among all such very free curves. Note that any curve satisfying these two conditions has the same curve class as $D$ since $b_2(X)=2$. We have
\begin{equation}\label{eq1}
-f^*K_Y \cdot D \leq -f^*K_Y \cdot C_n \leq - 2f^* K_Y \cdot C. 
\end{equation}
Now we want to show that this curve $D$ gives a non-zero Gromov-Witten invariant of the form $\langle [pt], [pt] \rangle_{0, D}^{X}$. 

We first claim that if there is an irreducible curve $F$ through the two points, then $-f^*K_Y \cdot F \geq -f^*K_Y \cdot D$ and equality holds if and only if they have the same curve class. Actually, if $-K_X \cdot F=4$, then this follows from our choice of $D$. If $-K_X \cdot F \geq 5$,  then we may deform this curve with two points fixed and we will end up with the same situations as above. We only need to deal with the case where the two points are not connected by an irreducible curve. But in this case we know that the two irreducible components have to be in the curve class $C$ by ~[\ref{eq1}]. But we may choose the two points so that any two curves in class $[C]$ through these two points do not meet. Then there has to be curves in the exceptional divisor and these curves are all the remaining curves. Note that we have
\[
A \cdot D \leq A \cdot (2C+L) \leq A \cdot F \leq A \cdot D,
\]
where $L$ is a line in the exceptional divisor. Therefore $[D]=[F]$. This is impossible since they have different intersection number with $-K_X$. The above argument also shows that no reducible curve of class $[D]$ could contribute to the Gromov-Witten invariant. Thus we are done.

Now assume the first case never happens for the any degeneration. We have $-f^*K_Y \cdot D_i \geq -f^*K_Y \cdot C$ and $A \cdot D_i \geq A \cdot C$ for $i=1, 2$. So the only components that are not in the exceptional divisor $E$ are the two components through the two points. And they have to have the same class as $C$. Then the curve in $E$ has to be a line since the intersection number with $-K_X$ is $5$. So this degeneration is just in the boundary of an irreducible component of expected dimension. We can alway choose the other constraint to miss such configurations.

\bibliographystyle{alpha}
\bibliography{MyBib}

\end{document}